\documentclass[a4paper,12pt,reqno]{amsart}
\usepackage{graphicx}

\usepackage{amsmath,amstext,amssymb,amsopn,amsthm}
\usepackage{url}

\usepackage[margin=30mm]{geometry}
\usepackage{eucal,mathrsfs,dsfont}%gives nicer set-names. Use \mathds instead of \mathds
\usepackage{soul}

\newtheorem{theorem}{Theorem}[section]

\newtheorem{lemma}[theorem]{Lemma}

\theoremstyle{definition}

\newtheorem{remark}{Remark}

\title[Multiple solutions for BVPs with fractional Laplacian]
      {Multiple solutions for Dirichlet nonlinear BVPs involving fractional Laplacian}

\author[Tadeusz Kulczycki and Robert Sta\'nczy]{Tadeusz Kulczycki and Robert Sta\'nczy}


 
 \address{Tadeusz Kulczycki, Instytut Matematyki i Informatyki, Politechnika Wroc{\l}awska, ul. Wybrze\.ze Wyspia\'nskiego 27, 50-370 Wroc{\l}aw, Poland}
 \email{Tadeusz.Kulczycki@pwr.wroc.pl}
 
 \address{Robert Sta{\'n}czy, Instytut Matematyczny, Uniwersytet Wroc{\l}awski, pl. Grunwaldzki 2/4, 50-370 Wroc{\l}aw, Poland}
 \email{stanczr@math.uni.wroc.pl}

\thanks{R. Sta{\'n}czy has been partially supported by the Polish Ministry of Science project N N201 418839. T. Kulczycki has been partially supported by NCN grant no. 2011/03/B/ST1/00423.}

\begin{document}

\begin{abstract}
The existence of at least two solutions to superlinear integral equation in cone is proved using
the Krasnosielskii Fixed Point Theorem. The result is applied to the 
Dirichlet BVPs with the fractional Laplacian. 
\end{abstract}

\maketitle

\section{Introduction and motivation}

It is well known that the superlinear equation with $p>1$ on the real line
\begin{equation}\label{qur}
u=bu^p+u_0
\end{equation}
can have none, one or more solutions $u$ depending on 
the data $b>0$ and $u_0\ge 0$. 
For example, if we additionally assume that
\begin{equation}\label{qco}
bu_0^{p-1}<c_p
\end{equation}
for some constant 
\begin{equation}\label{cp}
c_p=\left((p-1)^\frac{1-p}{p}+(p-1)^\frac1p\right)^{-p}
\end{equation}
then the existence of at least two nonnegative solutions of (\ref{qur}) is guaranteed, since thus the minimum of the function $bu^{p-1}+u_0 u^{-1}$ is ascertained
to be smaller than the constant $1$.

In this paper we would like to show that this simple observation 
can be generalized if we replace power term $bu^p$ defined on the real line
with a power like nonlinearity in a Banach space under some additional, suitable conditions like coercivity and compactness on some cone in this Banach space. More 
specifically, we shall consider the equation in the cone $P$ in the 
Banach space $E$ with the norm $|\cdot|$ in the form
\begin{equation}\label{ebf}
u=B(u)+u_0
\end{equation}
for some given element $u_0\in P$ and p-power, coercive and compact  
form $B$ defined on $P$.
The assumption (\ref{qco}) guaranteeing the existence of at least two solutions 
for the quadratic equation (\ref{qur}) now has to be adequately rephrased for (\ref{ebf})  
as
\begin{equation}\label{cbf}
b|u_0|^{p-1}<c_p\
\end{equation}
where $b>0$ denotes the best estimate such that for any $u\in P$
\begin{equation}\label{Bpe}
|B(u)|\le b|u|^p\,.
\end{equation}
Our main theoretical tool for the application to the superlinear integral equations and the BVPs with the fractional Laplacian is the following theorem.
\begin{theorem}
\label{main}
Assume that, for any given cone $P\subset E$, a compact mapping $B:P\rightarrow P$ satisfies the following condition
\begin{equation}\label{coe}
a|u|^p\le |B(u)|\le b|u|^p\,,
\end{equation}
for some $b > a > 0$.
Then for any $u_0\in P$ as small as to satisfy (\ref{cbf}) the equation (\ref{ebf}) 
admits at least two solutions in P.
\end{theorem}

As a direct but nontrivial application of this result we shall obtain among other applications a multiplicity result for the following superlinear boundary value problem involving the one-dimensional fractional Laplacian.
\begin{eqnarray}
\label {frac1}
(-\Delta)^{\alpha/2} u(x)&=&(u(x))^p+h(x)\,, \;\; {\rm for} \;\; x\in (-1,1),\\
\label{frac2}
u(x)&=&0, \quad \quad \quad \quad \quad \quad \;\; {\rm for} \;\; |x|\ge 1.
\end{eqnarray}

We shall denote by $G_{(-1,1)}$ both the Green function and the Green operator corresponding to the Dirichlet linear problem on $(-1,1)$ for the fractional Laplacian (see Preliminaries). We say that $u: [-1,1] \to [0,\infty)$ is symmetric and unimodal on $[-1,1]$ iff $u(x) = u(-x)$ for all $x \in [-1,1]$, $u$ is nondecreasing on $[-1,0]$ and nonincreasing on $[0,1]$. $BC([-1,1])$ denotes the space of all bounded continuous functions $f:[-1,1] \to \mathbb R$ with the standard supremum norm over the interval $[-1,1]$.

\begin{theorem}
\label{multifractional}
Let $\alpha \in (1,2)$, $p > 1$ and $h \in BC([-1,1])$ be a nonnegative, symmetric and unimodal function on $[-1,1]$. Assume also that (\ref{cbf}) is satisfied where $u_0 = G_{(-1,1)} h$ and $B u = G_{(-1,1)} u^p$.
Then there exist at least two nonnegative weak solutions to the boundary value problem (\ref{frac1}-\ref{frac2}). Morevoer, if $h$ is regular enough, i.e. $h\in C^\gamma (-1,1)$ with $\gamma>2-\alpha$ then the solutions are classical. 
\end{theorem}

The proofs of the above theorems will be postponed to the next sections.

The motivation for the fractional Laplacian originates from multiple sources, among others from: Probability and Mathematical Finance as the infinitesimal generators of stable L\'evy processes (\cite{Ber, BGR, ref:article}), which play nowadays an important role in stochastic modeling in applied sciences and in financial mathematics, Mechanics  encountered in elastostatics as  Signorini obstacle problem in linear elasticity (\cite{CS}) and finally from Fluid Mechanics as quasi-geostrophic fractional Navier-Stokes equation, see \cite{CV, Vaz} and references therein and Phase Transitions as described in \cite{SV}. Let us also mention here that the result corresponding to Theorem \ref{main} for the equations involving bilinear form, corresponding to $p=2$, were proved by one of the authors of this paper in \cite{Stj} motivated by the Navier--Stokes equation (cf. \cite{CM}), the Boltzmann equation (cf. \cite{dPL}), the quadratic reaction diffusion equation (cf. \cite{HW}), the Smoluchowski coagulation equation (cf. \cite{Smo}) or the system modeling chemotaxis \cite{Stc} to name but a few. The problem of uniqueness of solutions for these equations attracted a lot of attention and only some partial results are known. In some cases nonuniqueness occurs and the existence of two solutions can be proved. Sometimes one of the solution is a trivial one and then the proof relies on finding a nontrivial one, which can be of lower regularity or a nonstable one. In these models one encounters another problem making our approach not feasible i.e. very common lack of compactness, thus if we would like to make our approach feasible we are forced to consider some truncated baby model
compatible with compact setting.

To prove the existence of two solutions we shall use the Krasnoselskii
Fixed Point Theorem, cf. \cite{Guo}, which allows us to obtain 
more solutions if the nonlinear operator has the required property
of ``crossing" identity twice, i.e. by the cone compression and the 
expansion on some appropriate subsets of the cone.

It should be noted that the problem of existence of multiple solutions of nonlinear equations was addressed by H. Amann in \cite{Ama} in ordered Banach spaces rather than analysed from topological point of view as in our approach.

\section{Preliminaries concerning fractional Laplacian}
Let $\alpha \in (0,2)$ and $u:{\mathbb R}^d \to {\mathbb R}$ be a measurable function satisfying
\begin{equation}
\label{intcondition}
\int_{{\mathbb R}^d} \frac{|u(x)|}{(1+|x|)^{d+\alpha}} \, dx < \infty.
\end{equation}
For such a function the fractional Laplacian can be defined as follows (cf. \cite{BB9}, page 61)
$$\left(-\Delta\right)^{\alpha/2}u(x)=c_{d,-\alpha}\lim_{\varepsilon\to 0^+} \int_{\{y\in \mathbb R^d:|x-y|>\varepsilon\}}\frac{u(x)-u(y)}{|x-y|^{d+\alpha}}dy,$$
whenever the limit exists. Here we have 
$$c_{d,\gamma}=\Gamma((d-\gamma)/2)/(2^\gamma \pi^{d/2}|\Gamma(\gamma/2)|)\,.$$

It is known that if $u$ satisfies (\ref{intcondition}) and $u \in C^2(D)$ for some open set $D \subset {\mathbb R}^d$ then $\left(-\Delta\right)^{\alpha/2}u(x)$ is well defined for any $x \in D$, which can be justified by Taylor expansion of the function $u$. The fractional Laplacian may also be defined in a weak sense, see e.g. page 63 in \cite{BB9}.

Let us consider the Dirichlet linear problem for the fractional Laplacian on a bounded open set $D \subset {\mathbb R}^d$
\begin{eqnarray}
\label {Dirichlet1}
(-\Delta)^{\alpha/2} u(x)&=& g(x), \quad x \in D,\\
\label{Dirichlet2}
u(x)&=&0, \quad \quad x \notin D.
\end{eqnarray}
It is well known that there exist the Green operator $G_D$ and the Green function $G_D(x,y)$ corresponding to the problem (\ref{Dirichlet1})-(\ref{Dirichlet2}). Namely, if $g \in L^{\infty}$ then the unique (weak) solution of this problem is given by
\begin{equation}
\label{Greenop}
u(x) = G_D g(y) = \int_D G_{D}(x,y) g(y) \, dy.
\end{equation}
It should be noted that this $u$ is in fact in $C^{\gamma}$ with $\gamma>0$, cf. \cite{ROS}, whence also follows that $G_D$ increases interior regularity by $\alpha$ on the level of the H\"older continuous functions. The definition and basic properties of the Green operator and the Green function may be found e.g. in \cite{BB9} or \cite{BB}. It is well known that for any $\alpha\in (0,2)$ the Green function for the ball $B(0,1)$ is given by an explicit formula \cite{BGR} 
$$G_{B(0,1)}(x,y)=c^d_{\alpha}|x-y|^{\alpha-d}\int_0^{w(x,y)}r^{\alpha/2-1}(r+1)^{-d/2}\,dr, \quad x,y \in B(0,1),$$
where
$$w(x,y)=(1-|x|^2)(1-|y|^2)|x-y|^{-2}$$
and
$$c^d_{\alpha}=\Gamma(d/2)/(2^{\alpha}\pi^{d/2}\Gamma^2(\alpha/2))\,.$$
We have $G_{B(0,1)}(x,y) = 0$ if $x \notin B(0,1)$ or $y \notin B(0,1)$. 

In \cite{ROS} some Krylov type estimates on the regularity of solutions to the equations involving fractional Laplacian were provided by X. Ros-Oton and J. Serra. The regularity and the existence and uniquness issues for the problems involving fractional Laplacian were also addressed by X. Cabr\'e and Y. Sire  \cite{CS1, CS2}. For any open bounded $C^{1,1}$ domain $D$, $g\in L^{\infty}$ and a distance function $\delta(x)=\text{dist}(x,\partial D)$ if $u$ is the solution of the Dirichlet problem (\ref{Dirichlet1})-(\ref{Dirichlet2}) then $u/\delta^{\alpha/2}|_{D}$ can be continuously extended to $\overline D$. Moreover, we have $u/\delta^{\alpha/2} \in C^{\gamma}(\overline{D})$ and we control the norm
$$||u/\delta^{\alpha/2}||_{C^{\gamma}(\overline D)} \le C |g|_{\infty}$$
for some $\gamma<\min\{\alpha/2,1-\alpha/2\}\,.$ It suffices, due to the compact embedding $C^{\gamma}(\overline{D})\subset C(\overline{D})$, for compactness of the operator
$$G_D:C(\overline{D})\to C(\overline{D}).$$

We say that the bounded measurable function $u:{\mathbb R}^d \to {\mathbb R}$ is $\alpha$-harmonic in an open set $D \subset {\mathbb R}^d$ if $(-\Delta)^{\alpha/2}u(x) = 0$, for any $x \in D$ (in the classical sense). It is known (see e.g. \cite{BB9}, \cite{BB}) that such a function $u$ satisfies
$$
u(x) = \int_{D^c} P_D(x,y) u(y) \, dy, \quad x \in D,
$$
where $P_D: D \times D^c \to {\mathbb R}$ is the Poisson kernel (corresponding to the fractional Laplacian). The Poisson kernel for a ball $B(0,r) \subset {\mathbb R}^d$, $r > 0$ is given by an explicit formula (\cite{BGR})
$$
P_{B(0,r)}(x,y) = C_{\alpha}^d \frac{(r^2 - |x|^2)^{\alpha/2}}{(|y|^2 - r^2)^{\alpha/2} |x-y|^d}, \quad |x| < r,  |y| > r,
$$
where $C_{\alpha}^d = \Gamma(d/2) \pi^{-d/2-1} \sin(\pi \alpha/2)$.

\section{The abstract multiplicity result for compact p-power operators}

To prove Theorem \ref{main} we shall 
follow the lines of the proof presented in \cite{Stj} for $p=2$ and use the following theorem \cite[Theorem 2.3.4]{Guo} originating from the works of Krasnoselskii, cf. \cite{Kra}.
\begin{theorem}\label{con}
Let $E$ be a Banach space, and let $P\subset E$ be a cone in $E$. Let $\Omega_1$ and
$\Omega_2$ be two bounded, open sets in $E$ such that $0\in \Omega_1$ and $\overline{\Omega}_{1}\subset\Omega_2$. Let completely continuous operator $T:P\rightarrow P$ satisfy conditions
$$
|Tu|\leq |u| {\rm \;\; for \;\; any \;\;} u\in P\cap\partial\Omega_1 {\rm \;\; and \;\;}
|Tu|\geq |u| {\rm \;\; for \;\; any \;\;} u\in P\cap\partial\Omega_2 {\rm \;\;}
$$
or, alternatively, the following two conditions
$$
|Tu|\geq |u| {\rm \;\; for \;\; any \;\;} u\in P\cap\partial\Omega_1 {\rm \;\; and \;\;}
|Tu|\leq |u| {\rm \;\; for \;\; any \;\;} u\in P\cap\partial\Omega_2 {\rm \;\;}
$$
are satisfied. Then $T$ has at least one fixed point in $P\cap(\overline{\Omega}_2\setminus
\Omega_1)$.
\end{theorem}

\begin{proof}[proof of Theorem \ref{main}] 
Let us define the operator 
\begin{equation}\label{Tde}
Tu=B(u)+u_0
\end{equation}
then we shall apply Krasnosielskii Theorem once as a cone-compression in the neighborhood of zero and secondly as a cone-expansion at infinity.

Notice that we have the following estimates
\begin{eqnarray}
\begin{array}{ll}
|Tu|\le |u_0|+b|u|^p,\\
|Tu|\ge |u_0|-b|u|^p,
\end{array}
\end{eqnarray}
where constant $b=|B|>0$ denotes the the smallest constant $b$ satisfying, for any $u\in P$, the inequality
$$
|B(u)|\le b|u|^p\,.
$$ 

Then we can assume by (\ref{cbf}) that there exists some intermediate value $\rho_2>0$ such that
\begin{equation}\label{r2e}
|u_0|+b\rho_2^p<\rho_2\,.
\end{equation}
Indeed as announced in the introduction for the real line superlinear problem the above equation is equivalent to 
\begin{equation}\label{ur1}
|u_0|\rho_2^{-1}+b\rho_2^{p-1}<1\,.
\end{equation}
while the minimum of the function $|u_0|\rho_2^{-1}+b\rho_2^{p-1}$ is attained at $\rho_2$ such that $\rho_2^p=|u_0|/(b(p-1))$ and the minimum value is equal to
\begin{equation}\label{val}
|u_0|^{1-1/p}\left((b(p-1))^{1/p}+b^{1/p}(p-1)^{(1-p)/p}\right)\,.
\end{equation}
Requiring the value (\ref{val}) to be smaller than one as in (\ref{ur1}) is equivalent to (\ref{cbf}) which thus implies the claim (\ref{r2e}).

Hence by (\ref{coe}) together with (\ref{r2e}) for any $u\in P$ and $|u|=\rho_2$ one has
\begin{equation}\label{2co}
|Tu|\le |u_0|+b|u|^p<\rho_2=|u|.
\end{equation}

Moreover, if $u_0=0$ then $u=0$ is a solution. Otherwise, if $u_0\neq 0$ then for sufficiently small $\rho_1$ such that $\rho_2>\rho_1>0$ and $b\rho_1^p+\rho_1<|u_0|$ for any $u\in P$ and 
$|u|=\rho_1$ one has
\begin{equation}\label{1co}
|Tu|\ge |u_0|-b\rho_1^p > \rho_1=|u|.
\end{equation}
Thus both conditions (\ref{2co}) and (\ref{1co}) can be accomplished if we assume $b\rho_1^p+\rho_1< |u_0| < \rho_2-b\rho_2^p.$

Finally, for sufficiently large values of $\rho_3>0$ and any 
$u\in P$ and $|u|=\rho_3$, due to the coercivity assumption (\ref{coe})
\begin{equation}
|B(u)|\ge a|u|^p\,,
\end{equation}
one obtains
\begin{equation}\label{3co}
|Tu|\ge a\rho_3^p -|u_0|> \rho_3=|u|.
\end{equation}
To be more specific $\rho_3$ has to be so large that $\rho_3>\rho_2$ and $|u_0|<a\rho_3^p-\rho_3.$ 

Combining (\ref{2co}) with (\ref{1co}) we get that the intersection of the cone P with the spheres of the  
radii $\rho_1$ and $\rho_2$ (in the $|\cdot|$ norm) is compressed 
while the one at the radii $\rho_2$ and $\rho_3$ is expanded yielding the 
desired two fixed points in each set. Note that it might be necessary to 
distinguish between $\rho_2$ used in both sets as to prevent both fixed 
points to coincide. 

\end{proof}

\begin{remark}
To guarantee (\ref{3co}) in fact it suffices to assume only that
$$ \frac{|B(u)|}{|u|}\rightarrow\infty {\rm \;\; as \;\;} |u|\rightarrow \infty$$
instead of the lower estimate for the $B(u)$ as in (\ref{coe}). 
\end{remark}
 
\section{Multiplicity result for superlinear integral operator involving p-power nonlinearity}

Consider, for some open nonempty domain $V \subset \mathbb{R}^d$, the following equation in the space $BC(\overline{V})$ of bounded and continuous functions defined as 
\begin{equation}\label{hie}
{G F} u + u_0 = u
\end{equation}
where $u_0\in BC(\overline{V}$) is given, $u\in BC(\overline{V})$ is the unknown and $G$ is some linear integral operator defined by 
\begin{equation}
{ G} f(x) = \int_V G(x,y) f(y)\, dy
\end{equation}
for some given kernel function $G: \overline{V} \times \overline{V} \to \mathbb{R}$ smooth enough to guarantee compactness of ${G}$ in $BC(\overline{V})$, while a nonlinear operator $F$ is defined for $p>1$ by
\begin{equation}\label{Fde}
F u (y) =(u(y))^p\,.
\end{equation}
Then the operator $B$ from Theorem \ref{main} can be defined as 
\begin{equation}
B={ G F}\,.
\end{equation}
Let us define for some given, nonempty and open set $U\subset\subset V$ (i.e. $U$ is such that $\overline{U}\subset V$) and some constant $\gamma_V>0$ the cone $P$ as
\begin{equation}
P=\{u \in BC(\overline{V}): u \ge 0, \inf_U  u \ge \gamma_V \sup_{\overline{V}} u\}\,.
\end{equation}

Assume that the kernel $G$ is positive on $V \times V$ and that for any $y\in \overline{V}$ the following property holds
\begin{equation}\label{Ges}
\inf_{x\in U} G(x,y) \ge \gamma_V \sup_{x\in \overline{V}}G(x,y)
\end{equation}
where $\gamma_V>0$ is independent of $y$. 

Then the cone $P$ is invariant under $G F$. Using standard arguments (see \cite{Stj}) when we apply Theorem \ref{main} for $B=G F$
we arrive at the following theorem.
\begin{theorem}
There exist at least two nonnegative solutions to the Hammerstein equation (\ref{hie}) provided the function $G$ is regular enough to guarantee the compactness of the corresponding operator and satisfies (\ref{Ges}), while ${F}$ is defined by (\ref{Fde}) for some $p>1$ and $u_0$ is small enough as to satisfy (\ref{cbf}).
\end{theorem}

Note that to guarantee compactness of $G F$ usually the domain $U$ is assumed to be bounded and the kernel $G$ smooth enough but also for unbounded $U$
some results on compactness of $G$ under stronger decay assumptions on $F$ than the pure power like form were established, e.g. in \cite{Sth}.

\section{Multiplicity result for fractional Laplacian}

In this section we prove Theorem \ref{multifractional}. First we need two auxiliary lemmas. Let us denote $V = (-1,1)$. Recall that $G_V$ is the Green function for the one-dimensional problem (\ref{frac1})-(\ref{frac2}), also denoting the corresponding Green operator.

\begin{lemma}
\label{Green}
Let $a \in (0,1)$, $U = (-a,a)$. There exists $\gamma_U>0$ such that for any $y \in V$ we have
$$
\inf_{x \in U} G_V(x,y) \ge \gamma_U G_V(0,y).
$$
\end{lemma}
\begin{proof}
For any $x, y \in V$ by \cite[Corollary 3.2]{BB} we have
\begin{eqnarray}
\label{lower}
&& c_{\alpha} \left(\frac{\delta^{\alpha/2}(x) \delta^{\alpha/2}(y)}{|x-y|}\wedge \delta^{\frac{\alpha-1}{2}}(x) \delta^{\frac{\alpha-1}{2}}(y)\right) \\
\label{upper}
&& \le G_V(x,y) 
\le C_{\alpha} \left(\frac{\delta^{\alpha/2}(x) \delta^{\alpha/2}(y)}{|x-y|}\wedge \delta^{\frac{\alpha-1}{2}}(x) \delta^{\frac{\alpha-1}{2}}(y)\right),
\end{eqnarray}
where $\delta(x) = {\rm dist}(x,\partial V)$, $a \wedge b = \min(a,b)$.

Let $x \in U$, $y \in V$ be arbitrary. By (\ref{lower}) we get
$$
G_V(x,y) \ge c_U (\delta^{\alpha/2}(y) \wedge \delta^{\frac{\alpha-1}{2}}(y))
= c_U \delta^{\alpha/2}(y).
$$

On the other hand by (\ref{upper}) for any $y \in V$ we have
\begin{eqnarray*}
G_V(0,y) &\le& C_{\alpha} \left(\frac{\delta^{\alpha/2}(0) \delta^{\alpha/2}(y)}{|y|}\wedge \delta^{\frac{\alpha-1}{2}}(0) \delta^{\frac{\alpha-1}{2}}(y)\right) \\
&=& C_{\alpha} (\delta^{\alpha/2}(y) |y|^{-1} \wedge \delta^{\frac{\alpha-1}{2}}(y)).
\end{eqnarray*}
Hence for $y \in (-1/2,1/2)$ we get
$$
G_V(0,y) \le C_{\alpha} \delta^{\frac{\alpha-1}{2}}(y) \le C_{\alpha} \le 2^{\alpha/2} C_{\alpha} \delta^{\alpha/2}(y).
$$
For $y \in (-1,1) \setminus (-1/2,1/2)$ we obtain
$$
G_V(0,y) \le C_{\alpha} \delta^{\alpha/2}(y) |y|^{-1} \le 2 C_{\alpha} \delta^{\alpha/2}(y).
$$
\end{proof}

\begin{lemma}\label{symuni}
Suppose that $f \in BC(\overline{V})$ is nonnegative, symmetric and unimodal on $\overline{V}$. Then $G_V f\in BC(\overline{V})$ is also symmetric and unimodal on $\overline{V}$.  
\end{lemma}
\begin{proof} Symmetry of $G_V f$ follows by an explicit formula for the Green function of an interval (see Preliminaries). Note also that $G_V f(-1) = G_V f(1) = 0$. It is well known (see e.g. \cite{ROS}) that $G_V f$ is continuous on $\overline{V}$. Now we show that $G_V f$ is nonincreasing on $(0,1)$. To this end take any $0<x<y<1$ and fix $z=\frac{x+y}{2}$ and set $r=1-z$. Define the interval  $W=(z-r,z+r)=(z-r,1)$. By \cite[p. 87]{BB9} and \cite[p. 318]{BB} , for any $w\in W$ we have
$$ G_V f(w)= G_W f(w) + \int_{V\setminus W} G_V f(v) P_W(w,v)\,dv$$
where $G_V$, $G_W$ are Green operators for $V$, $W$ (respectively), while $P_W$ is the Poisson kernel for $W$, all corresponding to the fractional Laplacian $(-\Delta)^{\alpha/2}$ (see Preliminaries).
Let $\hat{w} = 2z-w$ be the inversion of a point w in respect to a point z. Clearly, we have $\hat{x} =y$ and $\hat{y}=x$. Let us observe that
$$
\int_{V\setminus W} G_V f(v) P_W(y,v)\,dv \le \int_{V\setminus W} G_V f(v) P_W(x,v)\,dv\,.
$$
Indeed, it follows from the fact, that for any $v\in V\setminus W$ one has 
\begin{eqnarray*}
P_W(y,v)
&=& \frac{C_\alpha^{1}}{|v-y|} \frac{(r^2-|y-z|^2)^{\alpha/2}}{(|v-z|^2-r^2)^{\alpha/2}}\\
&=& \frac{C_\alpha^{1}}{|v-y|} \frac{(r^2-|x-z|^2)^{\alpha/2}}{(|v-z|^2-r^2)^{\alpha/2}}\\ 
&\le& \frac{C_\alpha^{1}}{|v-x|} \frac{(r^2-|x-z|^2)^{\alpha/2}}{(|v-z|^2-r^2)^{\alpha/2}} = P_W(x,v)\,.
\end{eqnarray*}
Next we shall show that
$$
G_W f(y) \le G_W f(x)\,.
$$
Note that the Green function $G_W$ satisfies the following symmetry properties for any $v \in W$
\begin{eqnarray}
\label{Green1}
G_W(\hat{y},\hat{v})&=& G_W(y,v),\\
\label{Green2}
G_W(\hat{y},v)&=&G_W(y,\hat{v}).
\end{eqnarray}
Put $W_+=(z,1)$ and $W_-=(2z-1,z)$. It follows that 
\begin{eqnarray*}
G_W f(y) &=& \int_W G_W (y,v) f(v)\, dv\\
&=& \int_{W_+} G_W (y,v) f(v)\, dv+\int_{W_-} G_W (y,v) f(v)\, dv\\
&=&\int_{W_+} G_W (y,v) f(v)\, dv+\int_{W_+} G_W (y,{\hat v}) f({\hat v})\, dv\,.
\end{eqnarray*}
Similarly using $x=\hat{y}$ one obtains
\begin{eqnarray*}
G_W f(x) &=& \int_W G_W (\hat{y},v) f(v)\, dv\\
&=&\int_{W_+} G_W (\hat y,v) f(v)\, dv+\int_{W_-} G_W (\hat y,v) f(v)\, dv\\
&=&\int_{W_+} G_W (\hat y,v) f(v)\, dv+\int_{W_+} G_W (\hat y, \hat v) f(\hat v)\, dv.
\end{eqnarray*}
Using the above relation and again (\ref{Green1})-(\ref{Green2}) we get
\begin{eqnarray*}
&& G_W f(y)-G_W f(x)\\
&=& \int_{W_+} (G_W (y,v)-G_W (\hat y,v)) f(v)\, dv+\int_{W_+} (G_W (\hat y, v)-G_W (y, v)) f(\hat v)\, dv\\
&=& \int_{W_+}(G_W (y,v)-G_W (\hat y,v)) (f(v)-f(\hat v))\, dv \le 0\,, 
\end{eqnarray*}
since for any $v\in W_+$ 
\begin{eqnarray*}
f(v)-f(\hat v) \le 0\,,\\
G_W(y,v)-G_W(\hat y, v)\ge 0\,,
\end{eqnarray*}
by Corollary 3.2 from \cite{Kul}. It follows that $G_V f$ is nonincreasing on $(0,1)$.
\end{proof}

\begin{proof}[proof of Theorem \ref{multifractional}]
The problem can be formulated as required
\begin{equation}
u= G_V F u+u_0
\end{equation}
where
\begin{equation}
u_0(x)=G_V h(x)
\end{equation}
and 
\begin{equation}
G_V f(x)=\int_V G_V(x,y)f(y)\,dy\;,\;\;\; F u(x)=u(x)^p,
\end{equation}
where $G_V(x,y)$ is the Green function for $V$.

Let $a \in (0,1)$, $U = (-a,a)$ and $\gamma_U$ be the constant from Lemma \ref{Green}. Let us define  for the given $a$ the cone $P$ in the space of bounded and continuous functions $BC(\overline{V})$:
\begin{equation*}
P=\{u \in BC(\overline{V}): u \ge 0, \inf_U  u \ge \gamma_U \sup_{\overline{V}} u, \, \, \text{$u$ is symmetric and unimodal on $\overline{V}$}\}\,.
\end{equation*}
We will show that the cone $P$ is invariant under $B = G_V F$. Indeed, $B$ maps the set of bounded, continuous and nonnegative functions on $\overline{V}$ into itself.  Lemma \ref{symuni} gives that $B$ preserves symmetry and unimodality. What is more, for any $x \in U$ by Lemma \ref{Green} we have
\begin{eqnarray*}
B(u)(x) &=& \int_{-1}^1 G_V(x,y) u^p(y) \, dy \\
&\ge& \gamma_U \int_{-1}^1 G_V(0,y) u^p(y) \, dy = \gamma_U B(u)(0).
\end{eqnarray*}
It follows that $P$ is invariant under $B = G F$.

$B = G_V F$ satisfies the following coercivity condition with sup norms
\begin{eqnarray*}
\inf_{|u|=1, u \in P} |B(u)| &=&
\inf_{|u|=1, u \in P} \sup_{x \in V} \int_{-1}^1 G_V(x,y) u^p(y) \, dy \\
&\ge& \inf_{|u|=1, u \in P} \int_{-a}^a G_V(0,y) u^p(y) \, dy \\
&\ge& \inf_{|u|=1, u \in P} \int_{-a}^a G_V(0,y) \gamma_U^p |u|^p \, dy \\
&\ge& \gamma_U^{p} \int_{-a}^a G_V(0,y) \, dy > 0.
\end{eqnarray*}
We also have
\begin{eqnarray*}
\sup_{|u|=1, u \in P} |B(u)| &=&
\sup_{|u|=1, u \in P} \sup_{x \in V} \int_{-1}^1 G_V(x,y) u^p(y) \, dy \\
&\le&  \sup_{x \in V} \int_{-1}^1 G_V(x,y) \, dy < \infty.
\end{eqnarray*}
Hence $B:P \to P$ satisfies (\ref{coe}). Recall that the operator $B$ is compact (see Preliminaries).
Since (\ref{cbf}) is also satisfied Theorem \ref{main} gives that there exists at least two solutions in $P$ of
$$
u = B(u) + u_0.
$$
This equation may be rewritten as
$$
u = G_V(u^p + h). 
$$
Lemma 5.3 in \cite{BB} implies that the solution of this equation is a weak solution of (\ref{frac1})-(\ref{frac2}), which turns out due to the classical bootstrap argument that it is a classical one if we assume the function $h$ to be H\"older regular of order $\gamma>2-\alpha$, cf. \cite{ROS}. So we finally proved that there exists at least two solutions of (\ref{frac1})-(\ref{frac2}).
\end{proof}

The global solvability of some related problem under different conditions guaranteeing the integral operator to be a global diffeomorphism was considered in \cite{Bfr}.

%%%%%%%%%%%%%%%%%%%%%%%%%%%%%%%%%%%%%%%%%%%%%%%%%%%%%%%%%%%%%%%%%%%%%

%%%%%%%%%%%%%%%%%%%%%%%%%%%%%%%%%%%%%%%%%%%%%%%%%%%%%%%%%%%%%%%%%%%%%

\begin{thebibliography}{99}

\bibitem{Ama} 
{\sc H. Amann}, 
On the number of solutions of nonlinear equations in ordered Banach spaces,  
{\it J. Funct. Anal.} {\bf 11} (1972), 346--384.

\bibitem{AOR}
{\sc R.P. Agarwal and D. O'Regan}, Existence theorem for single and multiple solutions to singular positone boundary value problems, {\it J. Differential Equations} {\bf 175} (2001), 393--414.

\bibitem{Bar} 
{\sc P. Baras},
Non-unicit\'e des solutions d'une equation d'\'evolution non-lin\'eaire, {\it Annales Facult\'e des Sciences Toulouse} {\bf  5} (1983), 287--302.

\bibitem{Ber}
{\sc J.  Bertoin}, 
L\'evy Processes, Cambridge Tracts in Math., Cambridge Univ. Press 1996.

\bibitem{BGR} 
{\sc R. M. Blumenthal, R. K. Getoor and D. B. Ray}, On the distribution of first hits for the symmetric stable processes, {\it Trans. Amer. Math. Soc.} {\bf 99} (1961), 540--554.

\bibitem{BB9}
{\sc K. Bogdan and T. Byczkowski}, 
Potential theory for the $\alpha$-stable Schr\"odinger operator on bounded Lipschitz domain,
{\it Studia Math.} {\bf 133} (1999), 53--92.

\bibitem{BB}
{\sc K. Bogdan and T. Byczkowski}, 
Potential theory of Schr\"odinger operator based on fractional laplacian, {\it Probablility and Mathematical Statistics}  {\bf 20} (2000), 293--335.

\bibitem{ref:article}
{\sc K. Bogdan, T. Byczkowski, T. Kulczycki, M. Ryznar, R. Song and Z. Vondracek}, 
Potential Theory of Stable Processes and its Extensions, Lecture Notes in Mathematics, Springer 2009.

\bibitem{Bfr}
{\sc D. Bors}
Global solvability of BVP involving fractional Laplacian, preprint.

\bibitem{CS1}
{\sc X. Cabr\'e and Y. Sire}, 
Nonlinear equations for fractional Laplacians I: Regularity, maximum principles, and Hamiltonian estimates, arXiv 1012.0867, 2010.

\bibitem{CS2}
{\sc X. Cabr\'e and Y. Sire}, 
Nonlinear equations for fractional Laplacians II: existence, uniqueness, and qualitative properties of solutions, to appear,{\it Trans. AMS}.

\bibitem{CS}
{\sc L. Caffarelli and L. Silvestre}, 
An extension problem related to the fractional Laplacian, {\it Comm. Partial Differential Equations} {\bf 32} (2007), 1245--1260.

\bibitem{CV}
{\sc L. Caffarelli and A. Vasseur},
Drift diffusion equations with fractional diffusion and the quasi-geostrophic equation, {\it Ann. Math.} {\bf 171} (2010), 1903-1930.

\bibitem{CM}
{\sc M. Cannone and Y. Meyer}, 
Littlewood--Paley decomposition and the Navier--Stokes equations, {\it Methods Appl. Anal.} {\bf 2} (1995), 307--319.

\bibitem{FPS} {\sc P. Fija\l kowski, B. Przeradzki and R. Sta\'nczy}, A nonlocal elliptic equation in a bounded domain, {\it Banach Center Publications} {\bf 66} (2004), 127--133.

\bibitem{Guo}
{\sc D. Guo and  V. Lakshmikantham}, {\em Nonlinear Problems in Abstract Cones}, Academic Press, Orlando, FL, 1988.

\bibitem{Ha}
{\sc  K. S. Ha and Y. H. Lee},  Existence of multiple posiitve solutions
of singular boundary value problems, {\it Nonlinear Analysis} {\bf 28} (1997), 1429--1438.

\bibitem{HW}
{\sc A.  Haraux and F. B. Weissler}, Non-unicit\'e pour un probl\'eme de Cauchy semi-lin\'eaire,
Nonlinear partial differential equations and their applications, Coll\'ege de France Seminar,
{\bf III 428}, 220--233, Paris, 1980/1981, Res. Notes in Math. {\bf 70}, Pitman, Boston,
Massachussets, London, 1982.

\bibitem{Kul}
{\sc T. Kulczycki}, 
Gradient estimates of q-harmonic functions of fractional Schr\"odinger operator, 
{\it Potential Analysis}  {\bf 39} (2013), 69--98.

\bibitem{Kra}
{\sc Krasnosielski}, Topological methods in the theory of nonlinear integral equations, translated by A. H. Armstrong, translation edited by J. Burlak, A Pergamon Press Book The Macmillan Co., New York, 1964.

\bibitem{Lee}
{\sc Y. H. Lee}, An existence result of positive solutions for singular superlinear boundary value problems and its applications, {\it J. Korean Math. Soc.} {\bf 34} (1997), 247--255.

\bibitem{NPV}
{\sc E. Di Nezza, G. Palatucci and E. Valdinoci}, 
Hitchhiker's Guide to the Fractional Sobolev Spaces, preprint

\bibitem{dPL} {\sc R. J. Di Perna and P.-L. Lions}, On the Cauchy problem for Boltzmann equations: global existence and weak stability, {\it Annals of Math.} {\bf 130} (1989), 321--366.

\bibitem{PS}
{\sc  B. Przeradzki and R.  Sta\'nczy}, Positive solutions for sublinear elliptic equations, {\it Colloq. Math.} {\bf 92} (2002), 141--151.

\bibitem{ROS}
{\sc  X. Ros-Oton and J. Serra}, 
The Dirichlet problem for the fractional Laplacian: regularity up to the boundary,
{\it J. Math. Pures Appl.}, to appear, available online, (2012).

\bibitem{RS1}
{\sc X. Ros-Oton and J. Serra}, Fractional Laplacian: Pohozhaev identity and nonexistence results, 
{\it C. R. Math. Acad. Sci.} {\bf 350} (2012), 505--508.

\bibitem{SV}
{\sc Y. Sire, E. Valdinoci}, Fractional Laplacian phase transitions and boundary reactions: a geometric inequality and a symmetry result, 
{\it Journal of Functional Analysis} {\bf 256} (2009), 1842--1864.
   
\bibitem{Smo}
{\sc M. Smoluchowski}, Drei Vortr\"age \"uber Diffusion, Brownsche Molekularbewegung und Koagulation von Kolloidteilchen, {\it Physik. Zeit.} {\bf 17} (1916), 557--571, 585--599.

\bibitem{Sth}
{\sc R. Sta\'nczy}, Hammerstein equations with an integral over a non-compact domain, {\it Annales Polonici Mathematici} {\bf 69} (1998), 49-60

\bibitem{Str}
{\sc R. Sta\'nczy}, Nonlocal elliptic equations, {\it Nonlinear Analysis} {\bf 47} (2001), 3579--3584.

\bibitem{Sta}
{\sc R. Sta\'nczy}, Positive solutions for superlinear elliptic equations, {\it Journal of Mathematical Analysis and Applications} {\bf 283} (2003), 159--166.

\bibitem{Stc}
{\sc R. Sta\'nczy},
On radially symmetric solutions of some chemotaxis system, {\it Banach Center Publications} {\bf 86} (2009), 1--10.

\bibitem{Stj}
{\sc R. Sta\'nczy}, Multiple solutions for equations involving bilinear, coercive and compact forms with applications to differential equations, 
{\it Journal of Mathematical Analysis and Applications} {\bf 405} (2013), 416--421.

\bibitem{Val}
{\sc E. Valdinocci},
From the long jump random walk to the fractional laplacian,
{\it Bol. Soc. Esp. Mat. Apl.} {\bf 49} (2009), 33--44.

\bibitem{Vaz}
{\sc J. L. V\'azquez}
Nonlinear Diffusion with Fractional Laplacian Operators, 
{\it Nonlinear Partial Differential Equations, Abel Symposia} {\bf 7} (2012), 271--298.

\bibitem{Wei} 
{\sc F. B. Weissler}, 
Asymptotic analysis of an ordinary differential equation and nonuniqueness for a semilinear partial differential equation, {\it Arch. Rational Mech. Anal.} {\bf 91}  (1985), 231--245.


\end{thebibliography}
\end{document}